\documentclass[11pt]{amsart}
\usepackage{color}
\usepackage{amsmath}
\usepackage{amssymb}
\usepackage{amsfonts}
\usepackage[latin1]{inputenc}

\DeclareMathSymbol{\subsetneqq}{\mathbin}{AMSb}{36}

\newtheorem{theorem}{Theorem}
\theoremstyle{plain}

\newtheorem{corollary}{Corollary}

\newtheorem{definition}{Definition}

\newtheorem{lemma}{Lemma}

\newtheorem{remark}{Remark}

\numberwithin{equation}{section}
\input{tcilatex}

\begin{document}
\title[Blowup Criterion for Navier-Stojkes equations ]{The role of the Besov space $\mathbf{B}_{\infty }^{-1,\infty }$%
in the control of the eventual explosion in finite time of the
regular solutions of the Navier-Stokes equations}
\author{Ramzi May}
\address{Département de Mathématiques, Faculté des Sciences de Bizerte,
Jarzouna 2001 Bizerte Tunisie.} \email{Ramzi.May@fsb.rnu.tn}
\date{9 avril 2009}

\begin{abstract}
This paper is essentially a translation from French of my article \cite{M1} published in 2003. Let $u\in C([0,T^{\ast }[;L^{3}(\mathbb{R}%
^{3})) $ be a maximal solution of the Navier-Stokes equations. We
prove that $u$ is $C^{\infty }$ on $]0,T^{\ast }[\times
\mathbb{R}^{3}$ and there exists a constant $\varepsilon _{\ast
}>0$ independent of $u$ such
that if $T^{\ast }$ is finite then, for all $\omega \in \overline{S(\mathbb{R%
}^{3})}^{B_{\infty }^{-1,\infty }},$ we have
$\overline{\lim_{t\rightarrow T^{\ast }}}\left\Vert u(t)-\omega
\right\Vert _{\mathbf{B}_{\infty }^{-1,\infty }}\geq \varepsilon
_{\ast }.\medskip $
\end{abstract}
 \maketitle

\section{Introduction and results} In this note, we consider the integral Navier-Stokes equations:%
\begin{equation*}\label{eq}
u(t)=e^{t\Delta }u_{0}+\mathbb{L}(\mathbb{P}\nabla .(u\otimes
u))(t),
\end{equation*}%
where $u_{0}(x)=(u_{01},$ $u_{02},u_{02})$ is a given initial data
satisfying the divergence free condition $\nabla .u_{0}=0$ and
$u(t,x)=(u_{1},u_{2},u_{3}),$ the velocity, is the unknown. The
operator $\mathbb{P=}\left( \mathbb{P}_{ij}\right) _{1\leq i,j\leq
d}$ is the Leray projector and $\mathbb{L}$ is the linear operator
defined by:%
\begin{equation*}
\mathbb{L(}f\mathbb{)(}t\mathbb{)=-}\int_{0}^{t}e^{(t-s)\Delta }f(s)ds.
\end{equation*}%
Here $\left( e^{t\Delta }\right) _{t>0}$ is the heat semi-group defined
throug the Fourier Transform $\mathcal{F}$%
\begin{equation*}
\mathcal{F}\left( e^{t\Delta }f\right) (\xi )=e^{-t\left\vert \xi
\right\vert ^{2}}\mathcal{F}\left( f\right) (\xi ).
\end{equation*}%
In the sequel, we denote by $\mathbf{L}_{\sigma }^{3}$ the space
of $f=(f_1,f_2,f_3)\in L^{3}(\mathbb{R}^{3})$ that $\nabla .f=0.$
It is well known (see \cite{K} and \cite{FLT}) that for any
initial data $u_{0}\in \mathbf{L}_{\sigma }^{3},$ the equation has a unique maximal solution $%
u\in C([0,T^{\ast }[;L^{3}(\mathbb{R}^{3})).$ Recently, by using
the Caffarelli, Kohn et Nirenberg criterion, P. G.
Lemari\'{e}-Rieusset \cite{L} have proved that such solution $u$
is smooth on $Q_{T^{\ast }}\equiv ]0,T^{\ast }[\times
\mathbb{R}^{3}.$ In this paper, we will give a direct and simple
proof of this result.
\medskip

\noindent \textbf{Hereafter, we suppose that the maximal existence time }$T^{\ast }$%
\textbf{\ of the solution }$u$\textbf{\ is finite.} The main purpose of this
short paper, is to study the behavior of the solution near blowup time $%
T^{\ast }.$ Let us first recall some known results in this
direction: J. Leray \cite{Ler} and Y. Giga \cite{G} proved that
for any $p$ in $]3,+\infty ]$ there exists a constant $c_{p}>0$
such that
\begin{equation*}
\left\Vert u(t)\right\Vert _{p}\geq ~c_{p}(T^{\ast }-t)^{\frac{1}{2}(\frac{3%
}{p}-1)}.
\end{equation*}%
For the limit case $p=3,$ H. Shor and W. Von Wahl \cite{SV} proved that the solution $%
u $ can not be extended to a continuous function from $[0,T^{\ast }]$ into $%
L^{3}(\mathbb{R}^{3}).$ Later, H. Kozono and H. Shor \cite{KS}
improved this result: they established that there exists a
constant $\varepsilon _{KS}>0$ such
that if $\lim_{t\rightarrow T^{\ast }}u(t)=u^{\ast }$ in $L^{3}(\mathbb{R}%
^{3})$ with respect to the weak topology, then%
\begin{equation*}
\overline{\lim_{t\rightarrow T^{\ast }}}\left\Vert u(t)\right\Vert
_{3}^{3}-\left\Vert u^{\ast }\right\Vert _{3}^{3}\geq \varepsilon _{KS}.
\end{equation*}%
As a consequence they deduced that $u\notin BV([0,T^{\ast }[,L^{3}(\mathbb{R}%
^{3})).$ Recently, L. Escauriaza, G. Seregin et V. \v{S}ver\'{a}k
\cite{ESV} have proved that if in addition the solution $u$
belongs to the
Leray-Hopf  energy space \ $\mathcal{L}_{T^{\ast }}=L^{\infty }([0,T^{\ast }[,L^{2}(%
\mathbb{R}^{3}))\cap L^{2}([0,T^{\ast }[,H^{1}(\mathbb{R}^{3}))$
then
\begin{equation*}
\overline{\lim_{t\rightarrow T^{\ast }}}\left\Vert u(t)\right\Vert
_{3}=\infty .
\end{equation*}%
In the present paper, we aim to study the behavior of the solution
$u$ in the limit space space $B_{\infty }^{-1,\infty
}(\mathbb{R}^{3})$ (we recall that for any $p\geq 3$ we have
$L^{p}(\mathbb{R}^{3})\subset B_{\infty }^{-1,\infty
}(\mathbb{R}^{3})).$

\bigskip Our main result reads as follows:

\begin{theorem}\label{th1}
There exists constant $\varepsilon _{\ast }>0$ independent on $u$
such that, for any vectorial distribution $\omega =(\omega
_{1},\omega _{2},\omega _{3}) $ in
$\overline{S(\mathbb{R}^{3})}^{B_{\infty }^{-1,\infty }},$ we have
\begin{equation*}
\overline{\lim_{t\rightarrow T^{\ast }}}\left\Vert u(t)-\omega \right\Vert
_{B_{\infty }^{-1,\infty }}\geq \varepsilon _{\ast }.
\end{equation*}
\end{theorem}

\begin{remark}
This result remains true if we replace the space
$L^{3}(\mathbb{R}^{3})$ by any Lebesgue space
$L^{p}(\mathbb{R}^{3})$ with $p\geq 3$ or any Sobolev space
$H^{s}(\mathbb{R}^{3})$ with $s\geq \frac{1}{2}$ \cite{M}.
\end{remark}

\begin{remark}
Theorem \ref{th1} jointed to Weak-Strong uniqueness result of W.
Von Wahl allows to prove that any Leray-Hopf weak solution (see
\cite{L} for the definition) to the Navier-Stokes equations
belonging to the space $C([0,T],B_{\infty }^{-1,\infty })$ is
regular on $]0,T]\times \mathbb{R}^{3}$ \cite{M}.
\end{remark}

The following result is a direct consequence of Theorem \ref{th1}

\begin{corollary}
The solution $u$ does not belong to the space $BV([0,T^{\ast }[;B_{\infty
}^{-1,\infty }).$
\end{corollary}

\begin{proof}
By using the embedding $L^{3}(\mathbb{R}^{3})\subset \overline{S(\mathbb{R}%
^{3})}^{B_{\infty }^{-1,\infty }}$ and  Theorem \ref{th1}, one can
easily construct by indiction an increasing sequence $(t_{j})_{j}$
in $]0,T^{\ast }[$ such that
\begin{equation*}
\forall j,~\left\Vert u(t_{j+1})-u(t_{j})\right\Vert _{B_{\infty
}^{-1,\infty }}\geq \varepsilon _{\ast }.
\end{equation*}%
Therefore,%
\begin{equation*}
\sum_{j}\left\Vert u(t_{j+1})-u(t_{j})\right\Vert _{B_{\infty }^{-1,\infty
}}=\infty ,
\end{equation*}%
which implies the desired   result.
\end{proof}

\section{Preliminaries} In this section, we recall some definitions and results that will
be useful
in the proof of Theorem \ref{th1}. First, we define the nonhomogeneous Besov spaces $%
B_{p}^{s,\infty }.$ To do so, we need to introduce the
Littlewood-Paley decomposition: Let $\varphi $ be in the Schwartz
class $S(\mathbb{R}^{3})$ such that its Fourier Transform
$\mathcal{F}(\varphi )$ is identically equal to $1$ on the ball
$B(0,1)$ and vanishes outside the ball $B(0,2).$ For
$j\in \mathbb{N},~k\in \mathbb{N}^{\ast }$ and $f\in S^{\prime }(\mathbb{R}%
^{3}),$ we set%
\begin{equation*}
S_{j}f\equiv \varphi _{j}\ast f,~\Delta _{k}f\equiv S_{k}f-S_{k-1}f
\end{equation*}%
where $\varphi _{j}\equiv 2^{3j}\varphi (2^{j}.).$ Hence, for any $f\in
S^{\prime }(\mathbb{R}^{3}),$ we have the identity%
\begin{equation*}
f=S_{0}f+\sum_{k\geq 0}\Delta _{k}f,
\end{equation*}%
which is called the Littlewood-Paley decomposition of $f.$ In the
sequel we often denote the operator $S_{0}$ by $\Delta _{0}.$

\begin{definition}
Let $s\in \mathbb{R}$ and $1\leq p\leq \infty .$ The Besov space $%
B_{p}^{s,\infty }$ is defined by:%
\begin{equation*}
B_{p}^{s,\infty }=\{f\in S^{\prime }(\mathbb{R}^{3});~~\left\Vert
f\right\Vert _{B_{p}^{s,\infty }}\equiv \sup_{k\in N}2^{sk}\left\Vert \Delta
_{k}f\right\Vert _{p}<\infty \}.
\end{equation*}%
The space $\tilde{B}_{p}^{s,\infty }$ is the closure of $S(\mathbb{R}^{3})$
in $B_{p}^{s,\infty }.$
\end{definition}

In order to study the pointwise product in the Besov space, we
will use the
following weak version of the Bony decomposition: For $f$ and $g$ in $%
S^{\prime }(\mathbb{R}^{3}),$ we define%
\begin{eqnarray*}
\pi _{0}(f,g) &=&\sum_{k=0}^{\infty }S_{k}f~\Delta _{k}g, \\
\pi _{1}(f,g) &=&\sum_{k=0}^{\infty }S_{k+1}f~\Delta _{k}g.
\end{eqnarray*}%
Formally,%
\begin{equation*}
fg=\pi _{0}(f,g)+\pi _{1}(g,f).
\end{equation*}%
\medskip
The following elementary lemma will play a crucial role for the
proof of Theorem \ref{th1}

\begin{lemma}\label{lemme1}
Let $s>0.$ The bilinear operators $\pi _{0}$ and $\pi _{1}$ are
bounded from
$B_{\infty }^{-1,\infty }\times B_{\infty }^{s+1,\infty }$ (respectively, $%
L^{\infty }(\mathbb{R}^{3})\times B_{\infty }^{s+1,\infty })$ into $%
B_{\infty }^{s,\infty }$ (respectively, $B_{\infty }^{s+1,\infty
}$).
\end{lemma}

\begin{proof} One can consult the book \cite{L} of P.G. Lemari\'{e}%
-Rieusset.
\end{proof}
The next lemma recalls an important regularizing property of the
heat kernel

\begin{lemma} \label{lemme2}
Let $T\in ]0,1],~\alpha \in \{1;2\}$ and $r\in \mathbb{R}.$ The
linear operator $L,$ defined by (), is continuous from $L^{\infty
}([0,T],B_{\infty }^{r,\infty })$ into $L^{\infty
}([0,T],B_{\infty }^{r+\alpha ,\infty })$ and its norm is bounded
by $CT^{\frac{2-\alpha }{2}}$ where $C$ is a constant independent
of $T.$
\end{lemma}

\begin{proof}
See for example \cite{C}.
\end{proof}

We conclude this section by setting a slightly modified version of
the well-known existence theorem of T. Kato:

\begin{theorem} \label{th2}
Let $v_{0}\in \mathbf{L}_{\sigma }^{3}.$ Then, there exists a
unique $T_{\ast }\equiv T_{K}^{\ast }(v_{0})\in ]0,\infty ]$ and a
unique solution $v\equiv S_{K}^{\ast }(v_{0})$ to the integral
Navier-Stokes equations with initial data $v_{0}$ belonging to the
space $\cap _{0<T<T_{\ast }}\mathbf{L}_{K}^{n}(Q_{T})$, where
$\mathbf{L}_{K}^{n}(Q_{T})$ is the space of function $w\in
C([0,T];\mathbf{L}_{\sigma }^{3}))$ satisfying $\sqrt{t}w\in C([0,T];C_{0}(%
\mathbb{R}^{3}))$ and $\lim_{t\rightarrow 0}\sqrt{t}\left\Vert
w(t)\right\Vert _{\mathbf{\infty }}=0.$ Moreover, $v$ is smooth on $%
]0,T_{\ast }[\times \mathbb{R}^{3},$ more precisely, $v\in \cap
_{j,i\in \mathbb{N}}C_{t}^{i}(]0,T_{\ast }[,\tilde{B}_{\infty
}^{j,\infty }).$ Finally, there exists a constant $\varepsilon
_{3}>0$ independent of $v_{0}$ such that
\begin{equation*}
T_{K}^{\ast }(v_{0})\geq \sup \{T\in ]0,1];~\left( 1+\left\Vert
v_{0}\right\Vert _{\mathbf{3}}\right)
\sup_{0<t<T}\sqrt{t}\left\Vert e^{t\Delta }v_{0}\right\Vert
_{\mathbf{\infty }}\leq \varepsilon _{3}\}.
\end{equation*}
\end{theorem}

An immediate consequence of this theorem is the following
important result
\begin{lemma} \label{lemme3}
Let $v_{0}\in \mathbf{L}_{\sigma }^{3}$. Set $v=S_{K}^{\ast
}(v_{0})$ and $T_{K}^{\ast }=T_{K}^{\ast }(v_{0}).$ Then, for any
$t_{0}\in ]0,T_{K}^{\ast }[,$ we have $T_{K}^{\ast
}(v(t_{0}))=T_{K}^{\ast }-t_{0}$ and $S_{K}^{\ast
}(v(t_{0}))=v(.+t_{0}).$ Moreover, if $0<T_{K}^{\ast
}-t_{0}\leq 1$ then%
\begin{equation}
I_{\ast }(v_{0},t_{0})\overset{\text{def}}{=}\left( 1+\left\Vert
v(t_{0})\right\Vert _{\mathbf{3}}\right) \sup_{0<t<T_{K}^{\ast
}(v_{0})-t_{0}}\sqrt{t}\left\Vert e^{t\Delta }\left(
v(t_{0})\right) \right\Vert _{\mathbf{\infty }}>\varepsilon _{3}.
  \label{E3}
\end{equation}
\end{lemma}

\section{Proof of Theorem \ref{th1}} We divide the proof into 3
steps:\medskip

\noindent \textbf{First step:} We claim that $T^{\ast
}=T_{K}^{\ast }(u_{0})$ and $u=S_{K}^{\ast }(u_{0})$ (this implies
in particular, thanks to Theorem \ref{th2}, that the solution $u$
is regular on $]0,T^{\ast }[\times \mathbb{R}^{3}$). The
uniqueness theorem of solutions to the Navier-Stokes equations in
the space $C([0,T];\mathbf{L}_{\sigma }^{3})$ [3] ensures that
$T^{\ast }\geq
T_{K}^{\ast }(u_{0})$ and $u=S_{K}^{\ast }(u_{0})$ on the interval $%
[0,T_{K}^{\ast }(u_{0})[.$ Thus, we conclude once we show that $T^{\ast
}\leq T_{K}^{\ast }(u_{0}).$ We argue by opposition, we suppose that $%
T_{K}^{\ast }(u_{0})<T^{\ast }.$ Hence, the set $S_{K}^{\ast
}(u_{0})\left( [0,T_{K}^{\ast }(u_{0})[\right) =u\left(
[0,T_{K}^{\ast }(u_{0})[\right) $ is relatively compact in the
space $\mathbf{L}_{\sigma }^{3}.$ Therefore, by
using the inequality%
\begin{equation*}
\forall f\in L^{3}(\mathbb{R}^{3}),~\sup_{s>0}\sqrt{s}\left\Vert
e^{s\Delta }f\right\Vert _{\infty }\leq C\left\Vert f\right\Vert
_{3}
\end{equation*}%
and the fact%
\begin{equation*}
\forall f\in L^{3}(\mathbb{R}^{3}),~\lim_{s\rightarrow 0}\sqrt{s}\left\Vert
e^{s\Delta }f\right\Vert _{\infty }=0,
\end{equation*}%
one can easily deduce that there exists $\lambda \in ]0,1[$ such
that, for all $t_{0}\in \lbrack 0,T_{K}^{\ast }(u_{0})[,$ we have
\begin{equation*}
\left( 1+\left\Vert S_{K}^{\ast }(u_{0})(t_{0})\right\Vert _{3}\right)
\sup_{0<t<\lambda }\sqrt{t}\left\Vert e^{t\Delta }S_{K}^{\ast
}(u_{0})(t_{0})\right\Vert _{\mathbf{\infty }}\leq \varepsilon _{3}.
\end{equation*}%
Choosing $t_{0}$ so that $0<T_{K}^{\ast }(u_{0})-t_{0}<\lambda ,$ we get $%
I_{\ast }(u_{0},t_{0})\leq \varepsilon _{3},$ which contradicts
 (\ref{E3}).
 \medskip

\noindent \textbf{Second step:} We will prove that for all $a\in ]0,T^{\ast
}[,~u\notin L^{\infty }([a,T^{\ast }[,L^{\infty }(\mathbb{R}^{3})).$ We
argue by opposition. Let $a\in ]0,T^{\ast }[$ such that $\mathcal{M}\equiv
\sup_{a\leq t<T^{\ast }}\left\Vert u(t)\right\Vert _{\infty }<\infty .$ Let $%
b\in \lbrack a,T^{\ast }[$ to be chosen later. Set $v_{0}=u(b)$ and $%
v=S_{K}^{\ast }(v_{0}).$ Using Lemma \ref{lemme3}, the Young
inequality and the fact
that the $L^{1}(\mathbb{R}^{3})$ norm of the kernel $K_{t}$ of the operator $%
e^{t\Delta }\mathbb{P}\nabla $ is equal to $\frac{C}{\sqrt{t}},$ we obtain,
for all $t$ in $[0,T_{K}^{\ast }(v_{0})[,$ the following estimates%
\begin{eqnarray*}
\left\Vert v(t)\right\Vert _{3} &\leq &\left\Vert v_{0}\right\Vert
_{3}+C\int_{0}^{t}\frac{\left\Vert v(s)\right\Vert _{\infty }\left\Vert
v(s)\right\Vert _{3}}{\sqrt{t-s}}ds \\
&\leq &\left\Vert v_{0}\right\Vert _{3}+2
C\mathcal{M}\sqrt{t}\sup_{0\leq
s\leq t}\left\Vert v(s)\right\Vert _{3} \\
&\leq &\left\Vert v_{0}\right\Vert _{3}+2
C\mathcal{M}\sqrt{T_{K}^{\ast
}(v_{0})}\sup_{0\leq s\leq t}\left\Vert v(s)\right\Vert _{3} \\
&=&\left\Vert v_{0}\right\Vert _{3}+2 C\mathcal{M}\sqrt{T^{\ast }-b}%
\sup_{0\leq s\leq t}\left\Vert v(s)\right\Vert _{3}.
\end{eqnarray*}%
Therefore, by taking $b$ closed enough to $T^{\ast }$, we get
\begin{equation*}
\mathcal{N}\equiv \sup_{0\leq s<T_{K}^{\ast }(v_{0})}\left\Vert
v(s)\right\Vert _{3}<\infty .
\end{equation*}%
In conclusion, for all $t_{0}$ in $[0,T_{K}^{\ast }(v_{0})[$, we have%
\begin{equation*}
I_{\ast }(v_{0},t_{0})\leq (1+\mathcal{N})\mathcal{M}\sqrt{T_{K}^{\ast
}(v_{0})-t_{0}},
\end{equation*}%
which contradicts (\ref{E3}). \medskip

\noindent \textbf{Third
step:} Let $\varepsilon >0.$ Suppose that there exists $\omega \in
S(\mathbb{R}^{3})$ such that
\begin{equation*}
\overline{%
\lim_{t\rightarrow T^{\ast }}}\left\Vert u(t)-\omega \right\Vert
_{B_{\infty }^{-1,\infty }}<\varepsilon .
\end{equation*}
Then, there exists $\delta _{0}\in ]0,T^{\ast }[$ so that
\begin{equation*}
\sup_{t\in \lbrack T^{\ast }-\delta_0 ,T^{\ast }[}\left\Vert
u(t)-\omega \right\Vert _{\mathbf{B}_{\infty }^{-1,\infty
}}<\varepsilon .
\end{equation*}
Let $\delta \in ]0,\delta _{0}[$ to be chosen later. Set
$w_{0}=u(T^{\ast }-\delta )$ and $w=S_{K}^{\ast }(w_{0}).$
According to Lemma \ref{lemme3} $T_{K}^{\ast }(w_{0})=\delta $ and
$w=u(.+$ $T^{\ast }-\delta ).$ Thus,
\begin{equation*}
\sup_{0<t<\delta }\left\Vert w(t)-\omega \right\Vert _{\mathbf{B}_{\infty
}^{-1,\infty }}<\varepsilon .
\end{equation*}%
Fix $s>0.$ Kato's Theorem ensures that $w\in C([0,\delta
[;\tilde{B}_{\infty
}^{s+1,\infty }).$ On the other hand, we have%
\begin{equation*}
w(t)=e^{t\Delta }w_{0}+\sum_{j=0}^{1}\mathbb{L}\left( \mathbb{P}\nabla .\pi
_{j}\left[ (w-\omega )\otimes w\right] \right) +\mathbb{L}\left( \mathbb{P}%
\nabla .\pi _{j}\left[ \omega \otimes w\right] \right) (t).
\end{equation*}%
Therefore, applying Lemma \ref{lemme1} and \ref{lemme2} and using
the fact that $\mathbb{P}\nabla $ maps boundly $B_{\infty
}^{r,\infty }$ into $B_{\infty }^{r-1,\infty }$ \ (
$r\in \mathbb{R}$), yields for all $\delta _{1}<\delta$%
\begin{equation}
\sup_{0<t<\delta_{1} }\left\Vert w(t)\right\Vert _{B_{\infty
}^{s+1,\infty }}\leq \left\Vert w_{0}\right\Vert _{B_{\infty
}^{s+1,\infty
}}+C\{\varepsilon +\left\Vert \omega \right\Vert _{\infty }\sqrt{\delta }%
\}\sup_{0<t<\delta_{1} }\left\Vert w(t)\right\Vert _{B_{\infty
}^{s+1,\infty }}, \label{e}
\end{equation}%
where the constant $C$ depends only on $s.$ Now, suppose that
$\varepsilon \leq \varepsilon _{\ast }=\frac{1}{4C}.$ Then, one
can choose $\delta $ small enough so that $C\{\varepsilon
+\left\Vert \omega \right\Vert _{\infty }\sqrt{\delta }\}\leq
\frac{1}{2}.$ Hence, the estimate (\ref{e}) implies that
$\sup_{0<t<\delta_{1} }\left\Vert w(t)\right\Vert _{B_{\infty
}^{s+1,\infty }}\leq 2\left\Vert w_{0}\right\Vert _{B_{\infty
}^{s+1,\infty }}.$ Now using the embedding $B_{\infty
}^{s+1,\infty }\hookrightarrow
L^{\infty }(\mathbb{R}^{3})$ and the fact that $\delta_{1} $ is arbitrary in $%
]0,\delta [,$ we get
\begin{equation*}
\sup_{t\in \lbrack T^{\ast }-\delta ;T^{\ast }[}\left\Vert
u(t)\right\Vert _{\infty }=\sup_{0<t<\delta }\left\Vert
w(t)\right\Vert _{\infty }<\infty ,
\end{equation*}%
which contradicts the conclusion of the second step. Then, we conclude that
for all $\omega \in S(\mathbb{R}^{3})$ we have
\begin{equation*}
\overline{\lim_{t\rightarrow T^{\ast }}}\left\Vert u(t)-\omega \right\Vert
_{B_{\infty }^{-1,\infty }}\geq \varepsilon _{\ast }.
\end{equation*}%
By density, this inequality remains true for all $\omega \in \overline{S(%
\mathbb{R}^{3})}^{B_{\infty }^{-1,\infty }}.\medskip $
\providecommand{\bysame}{\leavevmode\hbox
to3em{\hrulefill}\thinspace}

\bibliographystyle{amsplain}


\providecommand{\bysame}{\leavevmode\hbox
to3em{\hrulefill}\thinspace}
\providecommand{\MR}{\relax\ifhmode\unskip\space\fi MR }
\providecommand{\MRhref}[2]{%
  \href{http://www.ams.org/mathscinet-getitem?mr=#1}{#2}
} \providecommand{\href}[2]{#2}

\end{document}